\documentclass[10pt]{article}
\usepackage{amsfonts,amssymb,amsmath,comment}
\usepackage{graphicx,graphics,color}
\usepackage[english]{babel}

 \makeatletter
 \@addtoreset{equation}{section}
 \makeatother

 \newcounter{enunciato}[section]

 \newtheorem{ittheorem}{Theorem}
 \newtheorem{itlemma}{Lemma}
 \newtheorem{itproposition}{Proposition}
 \newtheorem{itdefinition}{Definition}
 \newtheorem{itcorollary}{Corollary}
 \newtheorem{itconjecture}{Conjecture}

 \newenvironment{theorem}{\addtocounter{enunciato}{1}
 \begin{ittheorem}}{\end{ittheorem}}

 \newenvironment{lemma}{\addtocounter{enunciato}{1}
 \begin{itlemma}}{\end{itlemma}}

\newcommand{\halmos}{\rule{1ex}{1.4ex}}

\newenvironment{proof}{\noindent {\em Proof}.\,\,}
   {\hspace*{\fill}$\halmos$\medskip}

\def \qed {{\hspace*{\fill}$\halmos$\medskip}}

\def \ba {\begin{array}}
\def \ea {\end{array}}

\def \Z {{\mathbb Z}}
\def \R {{\mathbb R}}

\def \N {{\mathbb N}}

\def \ra {\rightarrow}

\begin{document}
\title{Minimising Dirichlet eigenvalues on cuboids of unit measure}

\author{\renewcommand{\thefootnote}{\arabic{footnote}}
M.\ van den Berg
\footnotemark[1]
\\
\renewcommand{\thefootnote}{\arabic{footnote}}
K.\ Gittins
\footnotemark[2]
}

\footnotetext[1]{School of Mathematics, University of Bristol, University Walk, Bristol BS8 1TW, United Kingdom.\, {\tt mamvdb@bristol.ac.uk} }

\footnotetext[2]{School of Mathematics, University of Bristol, University Walk, Bristol BS8 1TW, United Kingdom.\, {\tt kg13951@bristol.ac.uk}}

\date{7 July 2016}

\maketitle

\begin{abstract}
We consider the minimisation of Dirichlet eigenvalues $\lambda_k$, $k \in \N$, of the Laplacian
on cuboids of unit measure in $\R^3$. We prove that any sequence of optimal cuboids
in $\R^3$ converges to a cube of unit measure in the sense of Hausdorff as $k \rightarrow \infty$.

\vskip 0.5truecm
\noindent
{\it AMS} 2010 {\it subject classifications.} 35J20, 35P99.

\medskip\noindent
{\it Acknowledgements.} MvdB acknowledges support by The Leverhulme Trust
through International Network Grant \emph{Laplacians, Random Walks, Bose Gas,
Quantum Spin Systems}. KG was supported by an EPSRC DTA. The authors wish to thank Trevor Wooley for helpful discussions.

\end{abstract}

\section{Introduction.}\label{Introduction}
The eigenvalues of the Laplacian have been the object of intensive study over the last century. Of particular interest
are related shape optimisation problems. For $k \in \N$, the goal is to optimise the $k$'th eigenvalue of the Laplacian with boundary conditions over a collection of open sets in $\R^m$. This  collection satisfies geometric constraints, such as fixed Lebesgue measure or fixed perimeter.

For an open set $\Omega \subset \R^m$, $m \geq 2$, of finite Lebesgue measure $|\Omega|$, we let $\lambda_k(\Omega)$,
$k \in \N$, denote the Dirichlet eigenvalues of the Laplacian on $\Omega$ which are strictly positive, arranged in non-decreasing order and counted with multiplicity:
\begin{equation*}
\lambda_1(\Omega) \leq \lambda_2(\Omega) \leq \lambda_3(\Omega) \leq \dots
\leq \lambda_k(\Omega) \leq \dots
\end{equation*}
This sequence accumulates at $+\infty$.

We consider the following minimisation problem:$$\lambda_k^*(m) := \inf \{\lambda_k(\Omega) : \Omega\, \textup{\,open in\,} \R^m,\,
\vert \Omega \vert = c\}.$$
It was shown by Faber and Krahn that among all open sets in $\R^{m}$ of measure $c$, the
ball of measure $c$ minimises the first Dirichlet eigenvalue, see \cite{H}.
Krahn and Szeg\"{o} proved that, among all open sets in $\R^{m}$ of measure $c$,
the second Dirichlet eigenvalue is minimised by the union of two disjoint balls of measure $\frac{c}{2}$ each, see \cite{H}.
For $k \geq 3$, the existence of an open set of prescribed measure which minimises the $k$'th Dirichlet eigenvalue
remains unresolved to date. However, in the class of quasi-open sets of prescribed measure, it was shown
by Bucur in \cite{B1} that a minimiser does exist and that such a minimiser is bounded and has finite perimeter.
Independently, Mazzoleni and Pratelli proved the existence of a minimiser in \cite{MP} in the collection of quasi-open sets.
For any lower semi-continuous, increasing function of the first $k$ Dirichlet eigenvalues, they proved the
existence of a minimiser which is bounded in terms of $k$ and $m$ independently of the function.
It was shown in \cite{vdBI} that for $k \leq m+1$, any bounded minimiser of $\lambda_k(\Omega)$ has
at most $\min\{7,k\}$ components.

No optimal domains are known for $\lambda_k$ with $k \geq 3$. In particular, the conjecture that
if $m=2$, then $\lambda_3(\Omega)$ is bounded from below by the third eigenvalue of the disc with
the same measure as $\Omega$ is open. There are no obvious candidates for minimisers of $\lambda_k$
with $k \geq 5$ in any dimension $m \geq 2$. It was shown in \cite{Be} that for $k \geq 5$,
$\lambda_k(\Omega)$ cannot be minimised by a disc or a disjoint union of discs. Numerical investigations
\cite{AF1, Ou} suggest that for some values of $k$ the minimisers may not have any symmetries or may
not be describable in terms of standard functions.

P\'{o}lya's conjecture for Dirichlet eigenvalues asserts that for all bounded, open sets
$\Omega \subset \R^m$, $\lambda_k(\Omega) \geq 4\pi^2 (\omega_m \vert \Omega \vert)^{-2/m} k^{2/m}$,
where $\omega_m$ denotes the measure of a ball in $\R^m$ of radius 1. It was shown in \cite{CoSo}
that P\'{o}lya's conjecture is equivalent to $\lambda_k^*(m)$ being asymptotically equal to $4\pi^2
(\omega_m c)^{-2/m} k^{2/m}$ as $k \rightarrow \infty$.

It is also interesting to consider the optimisation of the eigenvalues of the Laplacian subject to other geometric constraints, such as fixed perimeter. For the Dirichlet eigenvalues, existence of a minimiser in the class of open
sets in $\R^m$ of finite Lebesgue measure and prescribed perimeter was shown in \cite{DPV}. Moreover, it was shown there
that any minimiser is bounded and connected, and regularity results for the boundary were also obtained.
Bucur and Freitas, \cite{BF}, showed that any sequence of minimisers of $\lambda_k$ in $\R^2$ with perimeter $\ell$ converges in the sense of Hausdorff to the disc of perimeter $\ell$ as $k \rightarrow \infty$. They also showed that if the collection of admissible sets is restricted to the collection of $n$-sided, convex, planar polygons of perimeter $\ell$, then any
sequence of minimisers converges to the regular $n$-sided polygon of perimeter $\ell$ as $k \rightarrow \infty$.
For $m \geq 2$, other constraints were considered in \cite{vdB1}, including perimeter and moment of inertia,
subject to an additional convexity constraint. Further results for the Dirichlet eigenvalues were obtained in \cite{AF2}, \cite{BBH}, \cite{vdBI}, \cite{BF} and \cite{vdB1}.
Some of the results of \cite{AF2} follow directly from those in \cite{vdB1}, while the results of \cite{DPV} supersede those of \cite{BBH}.

Recently, Antunes and Freitas considered the problem of minimising
$\lambda_k$ over all planar rectangles of unit measure, \cite{AF}. They proved that
any sequence of minimising rectangles for the Dirichlet eigenvalues
converges to the unit square in the sense of Hausdorff as $k \rightarrow \infty$.

In this paper, we focus on the Dirichlet eigenvalues of the Laplacian on cuboids
in $\R^3$ of unit measure. By a cuboid in $\R^3$ we mean a parallelepiped in $\R^3$
where all faces of the boundary are rectangular.
Let $R_{a_1 , a_2 , a_3}$ denote a cuboid in $\R^3$ of side-lengths $a_1 , a_2, a_3$ such that
$a_1 a_2 a_3 = 1$ and $a_1 \leq a_2 \leq a_3$,
\begin{equation}\label{ei0}
R_{a_1,a_2,a_3} = \{(x_1,x_2,x_3) \in \R^3 : 0 < x_1 < a_1, 0 < x_2 < a_2, 0 < x_3 < (a_1 a_2)^{-1}\}.
\end{equation}

We prove the following.
\begin{theorem}\label{T1}
\begin{enumerate}
\item[\textup{(i)}] Let $k\in \N$. The  variational problem
\begin{equation*}
\lambda_k^*:=\inf\{\lambda_k(R_{a_{1}, a_{2}, a_{3}}) :
a_{1} \leq a_{2} \leq a_{3}\}
\end{equation*}
has a minimising cuboid $R_{a_{1,k}^*,a_{2,k}^*,a_{3,k}^*}$ with side-lengths $a_{1,k}^*\le a_{2,k}^*\le a_{3,k}^*$, such that $a_{1,k}^* a_{2,k}^* a_{3,k}^* =1.$ \\
 \item[\textup{(ii)}]
\begin{equation}\label{ei2a}
a_{3,k}^* \leq 1 + O(k^{-(2-\beta)/6}), \, k \rightarrow \infty,
\end{equation}
where $\beta$ is the exponent of the remainder in $$\# \{(i_1,i_2,i_3)\in \Z^3:i_1^2+i_2^2+i_3^2 \leq R^2\} -\frac{4\pi}{3}R^3=O(R^{\beta})\, , R \ra \infty. $$
Furthermore, any sequence of optimal cuboids $R_{a_{1,k}^*,a_{2,k}^*,a_{3,k}^*}$ converges to the unit cube in $\R^3$ in the sense of Hausdorff as $k \rightarrow \infty$.
\end{enumerate}
\end{theorem}

The best known estimate to date is $\beta > \frac{63}{43}$, see \cite{IKKN04}. Hence \eqref{ei2a} holds for $\beta= \frac{63}{43}+\epsilon$ where $\epsilon>0$ is arbitrary. Theorem~\ref{T1} extends the result which was given in
Theorem 2.1 of \cite{AF} for the planar case mentioned above.

The Dirichlet eigenvalues of the Laplacian on a cuboid $R_{a_1,a_2,a_3}$ (as in \eqref{ei0}) are given by
\begin{equation}\label{ei1}
\frac{\pi^2 i_1^2}{a_{1}^2} + \frac{\pi^2 i_2^2}{a_{2}^2} + \frac{\pi^2 i_3^2}{a_{3}^2}, \quad
i_1, i_2, i_3 \in \N.
\end{equation}
By listing these in non-decreasing order including multiplicities, the $k$'th Dirichlet eigenvalue on
$R_{a_1,a_2,a_3}$, $\lambda_k(R_{a_1,a_2,a_3})$, is the $k$'th item of this list. Let $\lambda \in \R$,
$\lambda \geq 0$, and $a_{1}, a_{2}, a_{3} \in \R$ such that $a_{1}a_{2}a_{3}=1$ and $a_{1}\leq a_{2} \leq a_{3}$.
With \eqref{ei1} in mind, we define
\begin{equation}\label{e2.1a}
E(\lambda) := \left\{(x_1,x_2,x_3) \in \R^{3} : \frac{x_1^2}{a_1^2} + \frac{x_2^2}{a_2^2} + \frac{x_3^2}{a_3^2}
\leq \frac{\lambda}{\pi^{2}} \right\}.
\end{equation}
The ellipsoid $E(\lambda)$ has semi-axes
\begin{equation*}
r_{1} = \frac{a_{1}\lambda^{1/2}}{\pi}, \quad r_{2} = \frac{a_{2}\lambda^{1/2}}{\pi}, \quad
r_{3} = \frac{a_{3}\lambda^{1/2}}{\pi},
\end{equation*}
and $\vert E(\lambda) \vert = \frac{4}{3\pi^{2}}\lambda^{3/2}$.

By \eqref{ei1} and \eqref{e2.1a}, we see that the Dirichlet eigenvalues $\lambda_1(R_{a_1,a_2,a_3}), \dots,
\lambda_k(R_{a_1,a_2,a_3})$ (counted with multiplicities) correspond to the integer lattice
points that are inside or on the ellipsoid $E(\lambda_k)$ in the first octant (excluding the
coordinate planes). Thus, in order to minimise $\lambda_k$ among all cuboids given by \eqref{ei0},
we wish to determine the $3$-dimensional ellipsoid $E(\lambda) \subset \R^3$ of minimal measure which
encloses $k$ integer lattice points in the first quadrant (excluding the coordinate planes).

For $n \in \N$, $n \geq 2$, estimates for the number of integer lattice points which are inside or on an
$n$-dimensional ellipsoid have been widely studied from a number theoretical viewpoint. However, in order to use
these estimates, it is crucial that the corresponding cuboids are bounded as $k \rightarrow \infty$.
As in the $2$-dimensional case, this is the most difficult part of the proof.

Following the strategy given in \cite{AF}, the first goal is to prove that the side-lengths of an optimal sequence of
cuboids $\big(R_{a_{1,k}^*, a_{2,k}^* ,a_{3,k}^*}\big)_k$ are uniformly bounded in $k$.  To do this, we obtain an upper bound for the
counting function $N(\lambda) = \# \{j \in \N : \lambda_j \leq \lambda\}$. This relies upon some key
lemmata which we prove in Section~\ref{S1}. In Section~\ref{bounded}, we prove that the optimal
cuboids are uniformly bounded. We are then able to use known estimates for the number of integer lattice
points inside and on an ellipsoid in Section~\ref{convergence}. This allows us to deduce Theorem~\ref{T1}(ii).

\section{Proof of Theorem~\ref{T1}(i).}\label{ProofT1i}

\begin{proof}
Fix $k \in \N$. Suppose that $\big\{R_{a_{1,k}^{(\ell)},a_{2,k}^{(\ell)},a_{3,k}^{(\ell)}}\big\}_{\ell \in \N}$ is a minimising sequence for $\lambda_{k}$ such that $a_{3,k}^{(\ell)} \rightarrow \infty$ as $\ell \rightarrow \infty$. In order to preserve the measure constraint $a_{1,k}^{(\ell)} \rightarrow 0$ as $\ell \rightarrow \infty$. So, we have that
\begin{equation*}
\lambda_{k}\big(R_{a_{1,k}^{(\ell)},a_{2,k}^{(\ell)},a_{3,k}^{(\ell)}}\big) > \frac{\pi^{2}}{(a_{1,k}^{(\ell)})^{2}}
\rightarrow \infty, \textup{ as } \ell \rightarrow \infty.
\end{equation*}
However, for the unit cube in $\R^3$, $\lambda_k \leq 3\pi^2 k^2<+\infty$.
This contradicts the assumption that $\big\{R_{a_{1,k}^{(\ell)},a_{2,k}^{(\ell)},a_{3,k}^{(\ell)}}\big\}_{\ell \in \N}$
is a minimising sequence for $\lambda_{k}$. So any minimising sequence $\big\{R_{a_{1,k}^{(\ell)},a_{2,k}^{(\ell)},a_{3,k}^{(\ell)}}\big\}_{\ell \in \N}$ for $\lambda_{k}$ is such that
$a_{1,k}^{(\ell)}, a_{2,k}^{(\ell)}, a_{3,k}^{(\ell)}$ are bounded as $\ell \rightarrow \infty$. Hence,
for each $i \in \{1,2,3\}$, there exists a convergent subsequence, again denoted by
$a_{i,k}^{(\ell)}$ such that  $a_{i,k}^{(\ell)}\rightarrow a_{i,k}^*$ for some $a_{i,k}^*\in (0,\infty)$.
Since $(a_{1},a_{2},a_{3})\mapsto \lambda_{k}(R_{a_{1}, a_{2}, a_{3}})$ is
continuous, $\lambda_k\big(R_{a_{1,k}^{(\ell)},a_{2,k}^{(\ell)},a_{3,k}^{(\ell)}}\big)\rightarrow \lambda_k\big(R_{a_{1,k}^*,a_{2,k}^*, a_{3,k}^*}\big)$ as $\ell \rightarrow \infty$.
Hence $R_{a_{1,k}^*,a_{2,k}^*,a_{3,k}^*}$ is a minimising cuboid for $\lambda_k$.
\end{proof}

It is not difficult to see that the above argument can also be used to prove the existence of a minimising
cuboid for $\lambda_{k}$ in $\R^{m}$ with $m \geq 4$.

\section{Key lemmata to prove boundedness of an optimal cuboid.}\label{S1}
The following lemma is crucial in the arguments that follow.
\begin{lemma}\label{L1}
Let $y \geq 0$, $a \geq 0$. For $n \in \{1, 2\}$, we have that
\begin{equation}\label{e1a}
\sum_{i=1}^{\left \lfloor{\frac{y^{1/2}}{a}}\right \rfloor}(y-a^{2}i^{2})^{n/2}
\leq \frac{\sqrt{\pi}}{2a}\frac{\Gamma\left(\frac{n+2}{2}\right)}{\Gamma\left(\frac{n+3}{2}\right)}y^{(n+1)/2} - \frac12 y^{n/2}
+\frac{(2an)^{n/2}}{(n+2)^{(n+2)/2}}y^{n/4}.
\end{equation}
\end{lemma}

\begin{proof}
We have that
\begin{equation}\label{e2}
\sum_{i=1}^{\left \lfloor{\frac{y^{1/2}}{a}}\right \rfloor}(y-a^{2}i^{2})^{n/2}
= a^{n}\sum_{i=1}^{\left \lfloor{\frac{y^{1/2}}{a}}\right \rfloor}\bigg(\left(\frac{y^{1/2}}{a}\right)^2
-i^{2}\bigg)^{n/2}.
\end{equation}
Let $R=\frac{y^{1/2}}{a}$ and consider $\sum_{i=1}^{\left \lfloor{R}\right \rfloor}g(i)$
where
\begin{equation}\label{e2a}
g(i)=(R^{2}-i^{2})^{n/2}.
\end{equation}
Then, for $0 \leq i \leq R$, we have that
\begin{align*}
g'(i)&= -ni(R^{2}-i^{2})^{(n-2)/2} \leq 0, 
\\
g''(i) &= n(R^{2}-i^{2})^{(n-4)/2} ((n-1)i^{2} - R^{2}) \leq 0. 
\end{align*}
So $i\mapsto g(i)$ is decreasing on $[0,R]$ and, since $n=1$ or $n=2$, $g$ is also concave on $[0,R]$.
We note that since $g$ is decreasing, $\sum_{i=1}^{\left \lfloor{R}\right \rfloor} g(i)$
is the total area of the rectangles of width 1 and height $g(i)$, $i \in \left\{1, \dots, \left \lfloor{R}\right \rfloor\right\}$,
which are inscribed in the curve $g(x)$ for $0 \leq x \leq R$.
Due to the concavity of $g$ on $(0,R)$, we can bound
$\sum_{i=1}^{\left \lfloor{R}\right \rfloor} g(i)$ from
above by the area under $g$ minus the area of the inscribed triangles
which sit on top of the aforementioned rectangles. That is
\begin{equation}
\sum_{i=1}^{\left \lfloor{R}\right \rfloor} g(i)
\leq \int_{0}^{R} g(i) \, di
-\frac{1}{2}\sum_{i=1}^{\left \lfloor{R}\right \rfloor}(g(i-1)-g(i))
-\frac{1}{2}\left(R-\left \lfloor{R}\right \rfloor\right)
g\left(\left \lfloor{R}\right \rfloor \right). \label{e5}
\end{equation}
We have that
\begin{align}
\int_{0}^{R} g(i) \, di
&= R^{n+1}\int_{0}^{1}(1-t^2)^{n/2} \, dt \notag\\
&= \frac{R^{n+1}}{2}\int_{0}^{1}(1-s)^{n/2} \frac{1}{\sqrt{s}}\, ds
= \frac{\sqrt{\pi}}{2}\frac{\Gamma\left(\frac{n+2}{2}\right)}{\Gamma\left(\frac{n+3}{2}\right)}R^{n+1}, \label{e6a}
\end{align}
where we have used [3.191.3, 8.384.1, \cite {GR}].

We also have that
\begin{align}
&-\frac{1}{2}\sum_{i=1}^{\left \lfloor{R}\right \rfloor}
(g(i-1)-g(i))-\frac{1}{2}\left(R-\left \lfloor{R}\right \rfloor\right)
g\left(\left \lfloor{R}\right \rfloor \right) \notag\\
&=-\frac{1}{2}R^{n} + \frac{1}{2}(1 + \left \lfloor{R}\right \rfloor - R)
\big(R^{2} - \left \lfloor{R}\right \rfloor^{2}\big)^{n/2} \notag\\
&=-\frac{1}{2}R^{n} + \frac{1}{2}(1 + \left \lfloor{R}\right \rfloor - R)
(R + \left \lfloor{R}\right \rfloor)^{n/2}
(R - \left \lfloor{R}\right \rfloor)^{n/2} \notag\\
&\leq -\frac{1}{2}R^{n} +\frac{1}{2}(2R)^{n/2} \max_{0 \leq \beta < 1} (1-\beta)\beta^{n/2} \notag\\
&= -\frac{1}{2}R^{n} + \frac{(2n)^{n/2}}{(n+2)^{(n+2)/2}} R^{n/2}. \label{e7}
\end{align}
Combining \eqref{e2}, \eqref{e5}, \eqref{e6a} and \eqref{e7} gives \eqref{e1a}.
\end{proof}

Applying the previous lemma with $n=1$, $y=\frac{a_2^2}{\pi^2}\lambda$, and $a=\frac{a_2}{a_1}$,
we recover the result of Theorem 3.1 from \cite{AF}. Since $g$ (as in \eqref{e2a}) is decreasing
on $[0,\frac{y^{1/2}}{a}]$, the following holds for all $n \in \N$.
\begin{lemma}\label{C1}
Let $y\geq 0$, $a \geq 0$. For $n \in \N$, we have that
\begin{equation*}
\sum_{i=1}^{\left \lfloor{\frac{y^{1/2}}{a}}\right \rfloor} (y - a^{2}i^{2})^{n/2} \notag\\
\leq \int_{0}^{\frac{y^{1/2}}{a}} (y-a^{2}i^{2})^{n/2} \, di =
\frac{\sqrt{\pi}}{2a}\frac{\Gamma\left(\frac{n+2}{2}\right)}{\Gamma\left(\frac{n+3}{2}\right)}y^{(n+1)/2}.
\end{equation*}
\end{lemma}

\section{Uniform boundedness of an optimal cuboid.}\label{bounded}
With $E(\lambda)$ as defined in \eqref{e2.1a}, we define the counting function
\begin{equation*}
N(\lambda):= \# \{j \in \N : \lambda_j(R_{a_1,a_2,a_3}) \leq \lambda \}
= \# \{(i_{1}, i_{2},i_{3}) \in \N^{3} \cap E(\lambda)\}.
\end{equation*}
We now use the results of Section~\ref{S1} to obtain an upper bound for $N(\lambda)$.

\begin{lemma}\label{L2}
For $\lambda \geq 0$ and $a_1 \leq a_2 \leq a_3$, $E(\lambda)$, $N(\lambda)$ as above, we have that
\begin{equation}\label{e2.6}
N(\lambda) \leq \frac{\lambda^{3/2}}{6\pi^{2}} - \frac{\lambda}{8\pi a_{1}} +
\frac{\lambda^{1/2}}{16 a_{1}^{2}}.
\end{equation}
\end{lemma}

\begin{proof}
For $(i_1,i_2,i_3) \in \N^{3} \cap E(\lambda)$, we have that
\begin{equation*}
i_3 \leq \left\lfloor{ \left(\frac{a_3^2}{\pi^2}\lambda - \frac{a_3^2}{a_1^2}i_1^2 -\frac{a_3^2}{a_2^2}i_2^2\right)_{+}^{1/2}}\right \rfloor,
\end{equation*}
where ``$+$'' denotes the positive part.
Hence
\begin{align}
N(\lambda) &\leq \sum_{i_{1} \in \N} \sum_{i_{2} \in \N}
\left \lfloor{\left(\frac{a_{3}^{2}}{\pi^{2}}\lambda - \frac{a_{3}^{2}}{a_{1}^{2}}i_{1}^{2}
-\frac{a_{3}^{2}}{a_{2}^{2}}i_{2}^{2}\right)_{+}^{1/2}}\right \rfloor \label{e2.3a}\\
&\leq \sum_{i_{1}=1}^{\left \lfloor{a_{1}\frac{\lambda^{1/2}}{\pi}}\right \rfloor}
\sum_{i_{2}=1}^{\left \lfloor{a_{2}\left(\frac{\lambda}{\pi^{2}}-\frac{i_{1}^{2}}{a_{1}^{2}}\right)^{1/2}}\right \rfloor}
\left(\frac{a_{3}^{2}}{\pi^{2}}\lambda - \frac{a_{3}^{2}}{a_{1}^{2}}i_{1}^{2}
-\frac{a_{3}^{2}}{a_{2}^{2}}i_{2}^{2}\right)^{1/2}. \label{e2.3}
\end{align}
Applying Lemma~\ref{C1} with $y=\frac{a_{3}^{2}}{\pi^{2}}\lambda - \frac{a_{3}^{2}}{a_{1}^{2}}i_{1}^{2}$,
$a=\frac{a_{3}}{a_{2}}$, $n=1$ to \eqref{e2.3}, we have that
\begin{equation}
N(\lambda)\leq \sum_{i_{1}=1}^{\left \lfloor{a_{1}\frac{\lambda^{1/2}}{\pi}}\right \rfloor}
\frac{\pi a_{2}}{4a_{3}} \left(\frac{a_{3}^{2}}{\pi^{2}}\lambda - \frac{a_{3}^{2}}{a_{1}^{2}}i_{1}^{2}
\right)
= \sum_{i_{1}=1}^{\left \lfloor{a_{1}\frac{\lambda^{1/2}}{\pi}}\right \rfloor} \frac{\pi a_{2}a_{3}}{4}
\left(\frac{\lambda}{\pi^{2}} - \frac{i_{1}^{2}}{a_{1}^{2}}\right).
\label{e2.4}
\end{equation}
Applying Lemma~\ref{L1} with $y=\frac{\lambda}{\pi^{2}}$, $a=\frac{1}{a_{1}}$, $n=2$, we obtain that
\begin{align}
\frac{\pi a_{2}a_{3}}{4} \sum_{i_{1}=1}^{\left \lfloor{a_{1}\frac{\lambda^{1/2}}{\pi}}\right \rfloor}
\left(\frac{\lambda}{\pi^{2}} - \frac{i_{1}^{2}}{a_{1}^{2}}\right)
&\leq \frac{\pi a_{2}a_{3}}{4}\left(\frac{2a_{1}}{3\pi^{3}}\lambda^{3/2} -\frac{1}{2\pi^{2}}\lambda
+\frac{1}{4\pi a_{1}}\lambda^{1/2}\right) \notag\\
&=\frac{\lambda^{3/2}}{6\pi^{2}} - \frac{\lambda}{8\pi a_{1}} +\frac{\lambda^{1/2}}{16 a_{1}^{2}}.
\label{e2.5}
\end{align}
By \eqref{e2.4} and \eqref{e2.5}, \eqref{e2.6} follows.
\end{proof}

We now prove that the side-lengths $a_{1,k}^{*}, a_{2,k}^{*}, a_{3,k}^{*},$ of an
optimal cuboid $R_{a_{1,k}^*,a_{2,k}^*,a_{3,k}^*}$ in $\R^3$ are uniformly bounded.

\begin{lemma}\label{L2.1}
For all $k \in \N$,
\begin{equation*}
a_{3,k}^* \leq 319.
\end{equation*}
\end{lemma}

\begin{proof}
Since \eqref{e2.6} holds for all $\lambda \geq 0$ and all cuboids, it holds for $\lambda = \lambda_{k}^*$
and an optimal cuboid $R_{a_{1,k}^*,a_{2,k}^*,a_{3,k}^*}$, so
\begin{equation*}
k \leq N(\lambda_k^*) \leq \frac{(\lambda_{k}^*)^{3/2}}{6\pi^{2}} - \frac{\lambda_{k}^*}{8\pi a_{1,k}^*} +
\frac{(\lambda_{k}^*)^{1/2}}{16 (a_{1,k}^*)^{2}},
\end{equation*}
and, by rearranging, we obtain that
\begin{equation}\label{e2.9}
\frac{(\lambda_{k}^{*})^{3/2}-6\pi^{2}k}{6\pi^{2}\lambda_{k}^{*}} \geq
\frac{1}{8\pi a_{1,k}^{*}} - \frac{(\lambda_{k}^{*})^{-1/2}}{16 (a_{1,k}^{*})^{2}}.
\end{equation}
The left-hand side of \eqref{e2.9} is an increasing function of $\lambda_{k}^{*}$, so it is bounded from
above by $\frac{\nu_{k}^{3/2}-6\pi^{2}k}{6\pi^{2}\nu_{k}}$, where $\nu_{k}$ is the $k$'th Dirichlet eigenvalue
of the Laplacian on the unit cube in $\R^3$. We obtain a lower bound for the right-hand side of $\eqref{e2.9}$
by using the fact that
\begin{equation*}
\lambda_{k}^{*} \geq \lambda_{1}(R_{a_{1,k}^*,a_{2,k}^*,a_{3,k}^*}) \geq \frac{\pi^{2}}{(a_{1,k}^{*})^{2}},
\end{equation*}
implies that
\begin{equation}
-\frac{(\lambda_{k}^{*})^{-1/2}}{16(a_{1,k}^{*})^{2}} \geq -\frac{1}{16\pi a_{1,k}^{*}}. \label{e2.12}
\end{equation}
Hence, by \eqref{e2.12}, we have that
\begin{equation*}
\frac{\nu_{k}^{3/2}-6\pi^{2}k}{6\pi^{2}\nu_{k}} \geq
\frac{1}{16\pi a_{1,k}^{*}},
\end{equation*}
which implies that,
\begin{equation}\label{e2.15}
a_{1,k}^{*} \geq \frac{1}{16\pi} \frac{6\pi^{2}\nu_{k}}{\nu_{k}^{3/2}-6\pi^{2}k}.
\end{equation}

We now obtain a uniform lower bound for $a_{1,k}^*$. Let $\omega_3$ denote the measure of a ball of radius 1
in $\R^3$. Then, by an estimate of Gauss, we have that
\begin{equation*}
N(\nu_k) = \# \bigg\{(i_1,i_2,i_3) \in \N^3 : i_1^2 + i_2^2 + i_3^2 \leq \frac{\nu_k}{\pi^2} \bigg\}
\geq \frac{\omega_3}{8}\bigg(\frac{\nu_k^{1/2}}{\pi} - 3^{1/2}\bigg)_{+}^{3}
\geq \frac{\nu_k^{3/2}}{6\pi^2} - \frac{3^{1/2} \nu_k}{2\pi}.
\end{equation*}
Let $\Theta_k$ denote the multiplicity of $\nu_k$. Then $N(\nu_k) \leq k + \Theta_k -1$. In addition,
$\Theta_k = \# \{(i_1,i_2,i_3) \in \N^3 : i_1^2 + i_2^2 + i_3^2 = \frac{\nu_k}{\pi^2}\}$
is the number of integer lattice points in the first octant that lie on the sphere in $\R^3$ which is centred at
$(0,0,0)$ and has radius $\frac{\nu_k^{1/2}}{\pi}$. By projection onto the plane $i_3 = 0$, each of these lattice
points corresponds to an integer lattice point which lies inside or on the circle $\{(i_1,i_2) \in \Z^2 : i_1^2 + i_2^2 = \frac{\nu_k}{\pi^2}\}$ in the first quadrant. The number of integer lattice points which lie inside or on this circle
is bounded from above by $\frac{\nu_k}{4\pi}$, i.e. the area inscribed by the circle in the first quadrant.
Thus we obtain that
\begin{equation}\label{e2.15b}
\nu_k^{3/2} \leq 6\pi^2 k + 3\pi \nu_k \bigg(\frac12 +3^{1/2}\bigg).
\end{equation}
Hence by \eqref{e2.15} and \eqref{e2.15b}, we have that
\begin{equation}\label{e2.15c}
a_{1,k}^* \geq \bigg(8\bigg(\frac12 + 3^{1/2}\bigg)\bigg)^{-1}.
\end{equation}
Using that $a_{1,k}^{*}\leq a_{2,k}^{*}\leq a_{3,k}^{*}$, $a_{1,k}^{*}a_{2,k}^{*}a_{3,k}^{*}=1$
and \eqref{e2.15c}, we deduce that
\begin{equation*}
a_{3,k}^* \leq \frac{1}{(a_{1,k}^*)^2} \leq 64 \bigg(\frac12 + 3^{1/2}\bigg)^2 \leq 319.
\end{equation*}
\end{proof}

The main obstructions to proving a corresponding result to Theorem~\ref{T1}(ii) in higher dimensions $m \geq 4$ are the following: (i) For $m \geq 4$ the corresponding upper bound for $N(\lambda)$ to \eqref{e2.3a} involves lattice point sums
$\sum_{i=1}^{\left \lfloor{R}\right \rfloor}g(i)$ with $g(i),\, R$ as in \eqref{e2a} and $n \geq 3$. For $n \geq 3$, $\frac{y^{1/2}}{a\sqrt{n-1}}$ is an inflection point of $g$ in $(0,\frac{y^{1/2}}{a})$ and so $g$ is not concave on $(0,\frac{y^{1/2}}{a})$. Thus, the above approach cannot be used to obtain an upper bound for the left-hand side of \eqref{e1a} when $n \geq 3$. (ii) The higher-dimensional equivalent of \eqref{e2.6} will contain more terms in the
right-hand side. The leading term in that right-hand side is the Weyl term. However, the lower order terms are bounds
which are uniform in $a_1$, for example. Their usefulness depends on the numerical coefficients which show up.
These in turn depend on lower dimensional lattice point sums.

\section{Proof of Theorem~\ref{T1}(ii).}\label{convergence}

The minimisers $R_{a_{1,k}^*,a_{2,k}^*,a_{3,k}^*}$ of $\lambda_k$ need not be unique.
From this point onwards, we consider an arbitrary subsequence of minimisers denoted by
$\big(R_{a_{1,k}^*,a_{2,k}^*,a_{3,k}^*}\big)_k$.

For $E(\lambda)$ as defined in \eqref{e2.1a}, we introduce the following notation.
\begin{align*}
T(\lambda) &= \# \{(x_{1},x_{2},x_{3}) \in \Z^{3} \cap E(\lambda) \}, \\ 
T_{x_{1}}(\lambda) &= \# \{(0,x_{2},x_{3}) \in  (\{0\} \times \Z^{2}) \cap E(\lambda) \}, \\ 
T_{x_{1}}^{+}(\lambda) &= \# \{(0,x_{2},x_{3}) \in (\{0\} \times \N^{2}) \cap E(\lambda) \}. 
\end{align*}
$T(\lambda)$ is the total number of integer lattice points that are inside or on the ellipsoid
$E(\lambda)$ in $\R^3$. Similarly $T_{x_1}(\lambda)$ is the number of integer lattice points that
are inside or on the ellipse in $\R^2$ which is centred at $(0,0)$ and has semi-axes $\frac{a_2 \lambda^{1/2}}{\pi}$,
$\frac{a_3 \lambda^{1/2}}{\pi}$. $T_{x_1}^{+}(\lambda)$ is the number of these lattice points that lie in
the first quadrant (excluding the axes). $T_{x_2}(\lambda)$, $T_{x_2}^{+}(\lambda)$ etc. are defined similarly.
Thus, we have that
\begin{align*}
T(\lambda) &= 8N(\lambda) + 4T_{x_{1}}^{+}(\lambda) + 4T_{x_{2}}^{+}(\lambda) + 4T_{x_{3}}^{+}(\lambda) \notag\\
&\ \ \ +2\left \lfloor{\frac{a_{1}\lambda^{1/2}}{\pi}}\right\rfloor +2\left \lfloor{\frac{a_{2}\lambda^{1/2}}{\pi}}\right\rfloor
+2\left \lfloor{\frac{a_{3}\lambda^{1/2}}{\pi}}\right\rfloor + 1,
\end{align*}
which implies that
\begin{align*}
N(\lambda) &= \frac{1}{8}T(\lambda) -\frac{1}{2}T_{x_{1}}^{+}(\lambda)-\frac{1}{2}T_{x_{2}}^{+}(\lambda)
-\frac{1}{2}T_{x_{3}}^{+}(\lambda) \notag\\
&\ \ \ -\frac{1}{4}\left \lfloor{\frac{a_{1}\lambda^{1/2}}{\pi}}\right\rfloor-\frac{1}{4}\left \lfloor{\frac{a_{2}\lambda^{1/2}}{\pi}}\right\rfloor
-\frac{1}{4}\left \lfloor{\frac{a_{3}\lambda^{1/2}}{\pi}}\right\rfloor-\frac{1}{8}.
\end{align*}
In addition, we have that
\begin{equation*}
T_{x_{1}}(\lambda)  = 4T_{x_{1}}^{+}(\lambda)  + 2\left \lfloor{\frac{a_{2}\lambda^{1/2}}{\pi}}\right\rfloor
+ 2\left \lfloor{\frac{a_{3}\lambda^{1/2}}{\pi}}\right\rfloor +1,
\end{equation*}
which implies that
\begin{equation*}
T_{x_{1}}^{+}(\lambda)=\frac{1}{4}T_{x_{1}}(\lambda) -\frac{1}{2}\left \lfloor{\frac{a_{2}\lambda^{1/2}}{\pi}}\right\rfloor
-\frac{1}{2}\left \lfloor{\frac{a_{3}\lambda^{1/2}}{\pi}}\right\rfloor -\frac{1}{4},
\end{equation*}
and similarly for $T_{x_{2}}^{+}(\lambda), T_{x_{3}}^{+}(\lambda)$.
Thus, we obtain
\begin{align}\label{e2.20d}
N(\lambda) &= \frac{1}{8}T(\lambda) -\frac{1}{8}T_{x_{1}}(\lambda) -\frac{1}{8}T_{x_{2}}(\lambda)
-\frac{1}{8}T_{x_{3}}(\lambda) \notag\\
&\ \ \ + \frac{1}{4}\left \lfloor{\frac{a_{1}\lambda^{1/2}}{\pi}}\right\rfloor +\frac{1}{4}\left \lfloor{\frac{a_{2}\lambda^{1/2}}{\pi}}\right\rfloor
+ \frac{1}{4}\left \lfloor{\frac{a_{3}\lambda^{1/2}}{\pi}}\right\rfloor +\frac{1}{4}.
\end{align}
Below we use this expression for $N(\lambda)$ in order to prove Theorem~\ref{T1}(ii).

\noindent{\it Proof of Theorem~\ref{T1}(ii).}
By setting $\lambda = \lambda_k^*$ in \eqref{e2.20d} and considering an optimal cuboid
$R_{a_{1,k}^*,a_{2,k}^*,a_{3,k}^*}$, we have that
\begin{align}\label{e2.20e}
k \leq N(\lambda_k^*) &= \frac{1}{8}T(\lambda_k^*) -\frac{1}{8}T_{x_{1}}(\lambda_k^*) -\frac{1}{8}T_{x_{2}}(\lambda_k^*)
-\frac{1}{8}T_{x_{3}}(\lambda_k^*) \notag\\
&\ \ \ + \frac{1}{4}\left \lfloor{\frac{a_{1,k}^*(\lambda_k^*)^{1/2}}{\pi}}\right\rfloor +\frac{1}{4}\left \lfloor{\frac{a_{2,k}^*(\lambda_k^*)^{1/2}}{\pi}}\right\rfloor
+ \frac{1}{4}\left \lfloor{\frac{a_{3,k}^*(\lambda_k^*)^{1/2}}{\pi}}\right\rfloor +\frac{1}{4}.
\end{align}

By Lemma~\ref{L2.1}, the $\{a_{1,k}^*, a_{2,k}^*, a_{3,k}^*\}$ are uniformly bounded,
so it is possible to make use of known estimates for the number of integer lattice
points that are inside or on a $3$-dimensional ellipsoid or a $2$-dimensional ellipse.
In particular there exists $C<\infty$ such that for all $\lambda \geq 0$
\begin{equation}\label{er2a}
\frac{4}{3\pi^{2}}\lambda^{3/2} - C\lambda^{\beta/2} \leq T(\lambda) \leq \frac{4}{3\pi^{2}}\lambda^{3/2} + C\lambda^{\beta/2} +1,
\end{equation}
where $\beta$ is as defined in the Introduction.
Similarly there exists $D<\infty$ such that for all $\lambda \geq 0$
\begin{equation}\label{er2b}
\frac{a_{2}a_{3}}{\pi}\lambda - D\lambda^{\theta/2} \leq T_{x_{1}}(\lambda) \leq \frac{a_{2}a_{3}}{\pi}\lambda + D\lambda^{\theta/2} +1,
\end{equation}
where $\theta$ is the exponent of the remainder in Gauss' circle problem
$$\#\{(i_1,i_2)\in \Z^2:i_1^2+i_2^2 \leq R^2\}-\pi R^2=O(R^{\theta})\, , R \ra \infty. $$
The best known estimate to date is $\theta > \frac{131}{208}$, see the Introduction in \cite{H03}.
Hence the formula above holds for $\theta=\frac{131}{208}+\epsilon$ for any $\epsilon > 0$.
The corresponding inequalities to \eqref{er2b} also hold for $T_{x_{2}}(\lambda) , T_{x_{3}}(\lambda) $.
Using these inequalities and \eqref{e2.20e}, we obtain the following upper bound for $N(\lambda_k^*)$.
\begin{align}
k \leq N(\lambda_k^*)&\leq \frac{(\lambda_k^*)^{3/2}}{6\pi^{2}} -\frac{1}{8\pi}\left(\frac{1}{a_{1,k}^*} +\frac{1}{a_{2,k}^*} +\frac{1}{a_{3,k}^*}\right)
\lambda_k^* +\frac{C}{8}(\lambda_k^*)^{\beta/2} \notag\\
&\ \ \ +\frac{1}{4\pi}(a_{1,k}^*+a_{2,k}^*+a_{3,k}^*)(\lambda_k^*)^{1/2} +\frac{3D}{8}(\lambda_k^*)^{\theta/2} +\frac{3}{8}. \label{e2.22}
\end{align}
Rearranging \eqref{e2.22}, we obtain that
\begin{align*}
\frac{1}{a_{1,k}^{*}} +\frac{1}{a_{2,k}^{*}} +\frac{1}{a_{3,k}^{*}} &\leq 8\pi\left(\frac{(\lambda_{k}^{*})^{3/2}-6\pi^{2}k}{6\pi^{2}\lambda_{k}^{*}}\right)
+\pi C(\lambda_k^*)^{-(2 - \beta)/2} \notag\\
&\ \ \ +2(a_{1,k}^*+a_{2,k}^*+a_{3,k}^*)(\lambda_k^*)^{-1/2} + 3\pi D(\lambda_k^*)^{-(2 - \theta)/2} + 3\pi(\lambda_k^*)^{-1}. 
\end{align*}
Since $\frac{(\lambda_{k}^{*})^{3/2}-6\pi^{2}k}{6\pi^{2}\lambda_{k}^{*}}$ is an increasing function of $\lambda_{k}^{*}$, we can replace $\lambda_{k}^{*}$ by $\nu_{k}$, where $\nu_{k}$ is the $k$'th Dirichlet
eigenvalue of the Laplacian on the unit cube in $\R^3$. Thus, by P\'olya's Inequality $\lambda_k^* \geq (6\pi^2k)^{2/3}$, (\cite{P, U84}), we obtain
\begin{align}
\frac{1}{a_{1,k}^{*}} +\frac{1}{a_{2,k}^{*}} +\frac{1}{a_{3,k}^{*}} &\leq 8\pi\left(\frac{\nu_{k}^{3/2}-6\pi^{2}k}{6\pi^{2}\nu_{k}}\right)
+\pi C(\lambda_k^*)^{-(2 - \beta)/2} + 3\pi(\lambda_k^*)^{-1} \notag\\
&\ \ \ +2(a_{1,k}^*+a_{2,k}^*+a_{3,k}^*)(\lambda_k^*)^{-1/2} + 3\pi D(\lambda_k^*)^{-(2 - \theta)/2} \notag\\
&\leq 8\pi\left(\frac{\nu_{k}^{3/2}-6\pi^{2}k}{6\pi^{2}\nu_{k}}\right)
+\pi C(6\pi^2)^{-(2-\beta)/3}k^{-(2-\beta)/3} + 3\pi(6\pi^2)^{-2/3}k^{-2/3} \notag\\
&\ \ \ +2(a_{1,k}^*+a_{2,k}^*+a_{3,k}^*)(6\pi^2)^{-1/3}k^{-1/3} + 3\pi D(6\pi^2)^{-(2 - \theta)/3}k^{-(2-\theta)/3} \notag\\
&=8\pi\left(\frac{\nu_{k}^{3/2}-6\pi^{2}k}{6\pi^{2}\nu_{k}}\right) + O(k^{-(2-\beta)/3}). \label{e2.23c}
\end{align}
To obtain an upper bound for $\frac{\nu_{k}^{3/2}-6\pi^{2}k}{6\pi^{2}\nu_{k}}$
we proceed as follows. By \eqref{e2.20d} with $\lambda=\nu_k$ we have that
\begin{equation}\label{ep1.1}
N(\nu_k) = \frac18 T(\nu_k) - \frac38 T_{x_1}(\nu_k) + \frac34 \left\lfloor \frac{\nu_k^{1/2}}{\pi} \right \rfloor
+ \frac14.
\end{equation}
Since $a_1 = a_2 = a_3 = 1$, by \eqref{er2a} and \eqref{er2b}, we have that
\begin{equation}\label{ep1.2a}
\frac{4}{3\pi^{2}}\nu_k^{3/2} - C\nu_k^{\beta/2} \leq T(\nu_k),
\end{equation}
and
\begin{equation}\label{ep1.2b}
T_{x_{1}}(\nu_k) \leq \frac{\nu_k}{\pi} + D\nu_k^{\theta/2} +1,
\end{equation}
where $\beta$ and $\theta$ are as in \eqref{er2a}, \eqref{er2b}. Again let $\Theta_k$ denote the multiplicity of $\nu_k$.
Thus by \eqref{ep1.1}, \eqref{ep1.2a} and \eqref{ep1.2b}, we obtain a lower bound for $N(\nu_k)$:
\begin{equation}\label{ep1.2e}
k + \Theta_k -1 \geq N(\nu_k) \geq \frac{\nu_k^{3/2}}{6\pi^2} - \frac{C}{8}\nu_k^{\beta/2}
-\frac{3}{8\pi}\nu_k - \frac{3D}{8}\nu_k^{\theta/2} + \frac{3}{4\pi} \nu_k^{1/2} - \frac78,
\end{equation}
which implies that
\begin{align*}
\frac{\nu_k^{3/2} - 6\pi^2k}{6\pi^2 \nu_k} &\leq \frac{3}{8\pi} + \frac{C}{8}\nu_k^{-(2-\beta)/2}
+ \frac{3D}{8}\nu_k^{-(2-\theta)/2} - \frac{3}{4\pi} \nu_k^{-1/2} + \Theta_k \nu_k^{-1} - \frac18 \nu_k^{-1} \notag\\
&\leq \frac{3}{8\pi} + \frac{C}{8}\nu_k^{-(2-\beta)/2} + \frac{3D}{8}\nu_k^{-(2-\theta)/2} + \Theta_k \nu_k^{-1} \notag\\
&\leq \frac{3}{8\pi} + \frac{C}{8}(6\pi^2)^{-(2-\beta)/3}k^{-(2-\beta)/3}
+ \frac{3D}{8}(6\pi^2)^{-(2-\theta)/3}k^{-(2-\theta)/3} + \Theta_k \nu_k^{-1},
\end{align*}
by P\'olya's Inequality.

We have that $\Theta_k = \# \{(i_1,i_2,i_3) \in \N^3 : i_1^2 + i_2^2 + i_3^2 = \frac{\nu_k}{\pi^2}\}$
is the number of integer lattice points in the first quadrant that lie on the sphere in $\R^3$ which is centred at
$(0,0,0)$ and has radius $\frac{\nu_k^{1/2}}{\pi}$.
It is well known that $\# \{(x_1,x_2,x_3) \in  \Z^3 : x_1^2 + x_2^2 + x_3^2 = d\}=O(d^{\frac12 + o(1)})$.

The following proof is due to T.\ Wooley.
Let $n = d - x_3^2$. Now $\vert x_3 \vert \leq d^{1/2}$, so for $x_3 \in [-d^{1/2},d^{1/2}] \cap \Z$,
there are at most $2d^{1/2} + 1$ possible values of $n$. If $n=0$, then $x_1^2 + x_2^2 = 0$ has one
solution $(0,0) \in \Z^2$. Suppose that $n \neq 0$. Let $R(n)$ denote the number of pairs $(x_1,x_2) \in \Z^2$
such that $x_1^2 + x_2^2 = n$. Then
\begin{equation*}
\# \{(x_1,x_2,x_3) \in \Z^3 : x_1^2 + x_2^2 + x_3^2 = d\} = 1 + \sum_{\vert z \vert \leq d^{1/2}} R(d-z^2).
\end{equation*}
By Corollary 3.23 of \cite{NZM}, we have that
\begin{equation*}
R(n) = 4 \sum_{d \vert n, \, d>0, \, d \text{ odd}} \bigg(\frac{-1}{d}\bigg),
\end{equation*}
where the sum is taken over all positive, odd divisors of $n$ and $(\frac{-1}{d})$ is the quadratic
residue symbol.  Thus $R(n) \leq 4D(n)$, where $D(n)$ denotes the number of divisors of $n$.
By Theorem 8.31 of \cite{NZM}, for every $\epsilon >0$, there exists $n_{\epsilon}$ such that for $n > n_{\epsilon}$,
\begin{equation*}
D(n) < n^{(1+\epsilon) \log 2 / \log \log n},
\end{equation*}
which implies that $D(n) = O(n^{\epsilon})$. Therefore we obtain that
\begin{equation*}
\# \{(x_1,x_2,x_3) \in \Z^3 : x_1^2 + x_2^2 + x_3^2 = d\} = 1 + O\Big(\sum_{\vert z \vert \leq d^{1/2}} (d-z^2)^{\epsilon}\Big) = 1 + O(d^{1/2 + \epsilon}).
\end{equation*}

So $\Theta_k = O(\nu_k^{\frac12 + o(1)})$
and $\Theta_k \nu_k^{-1} = O(\nu_k^{-\frac12 +o(1)}) = O(k^{-\frac13 + o(1)})$. Thus we obtain
\begin{equation}\label{ep}
\frac{\nu_{k}^{3/2}-6\pi^{2}k}{6\pi^{2}\nu_{k}} \leq \frac{3}{8\pi} + O(k^{-(2-\beta)/3}).
\end{equation}

So by \eqref{e2.23c} and \eqref{ep}, we deduce that
\begin{equation}\label{e2.24}
\frac{1}{a_{1,k}^{*}} + \frac{1}{a_{2,k}^{*}} + \frac{1}{a_{3,k}^{*}} \leq 3 + O(k^{-(2-\beta)/3}),
\, k \rightarrow \infty.
\end{equation}

Furthermore, by the Arithmetic Mean -- Geometric Mean Inequality applied to
$\frac{1}{a_{1,k}^*} + \frac{1}{a_{2,k}^*}$, we have by \eqref{e2.24} that
\begin{equation*}
2(a_{3,k}^*)^{1/2} + \frac{1}{a_{3,k}^*} \leq 3 + O(k^{-(2-\beta)/3}), \, k \rightarrow \infty.
\end{equation*}
Let $a_{3,k}^* = 1 + \delta_{k}$ where $\delta_{k} > 0$. Then
\begin{equation*}
2(1+\delta_{k})^{3/2} + 1 \leq 3 + 3 \delta_{k} + O(k^{-(2-\beta)/3}), \, k \rightarrow \infty.
\end{equation*} Since $a_{3,k}^* \leq 319$, $\delta_k \leq 399$. Hence
$(1+\delta_k)^{3/2} \geq 1 + \frac32 \delta_k + \frac{3}{160}\delta_k^2$
for $0 < \delta_k \leq 399$, we deduce that $\delta_k \leq O(k^{-(2 - \beta)/6})$, $k \rightarrow \infty$.
As this estimate is independent of the subsequence $\big(R_{a_{1,k}^*,a_{2,k}^*,a_{3,k}^*}\big)_k$
we arrive at the conclusion of Theorem~\ref{T1}(ii).

\qed

We remark that the proof of Theorem~\ref{T1}(ii) came in two steps. First we obtained bounds for the
counting function of the Dirichlet eigenvalues of an arbitrary cuboid and a unit cube which are uniform
in $k, a_1$ and $k$ respectively (Section~\ref{bounded}).
Having established such boundedness of a minimising cuboid, we then used asymptotic formulae for
both the counting function and the multiplicity of the  Dirichlet eigenvalues of the unit cube
(Section~\ref{convergence}).

\bigskip


\end{document}